  \documentclass{amsart}
\usepackage{amssymb}
\usepackage{amsmath, amscd}
\usepackage{amsthm}
\usepackage{mathrsfs}
\usepackage{perpage}
\usepackage[all,cmtip]{xy}
\usepackage{setspace}
\usepackage{amsbsy}
\usepackage[foot]{amsaddr}
\usepackage{url}

\theoremstyle{plain}
\newtheorem{thm}{Theorem}[section]
\newtheorem{prop}[thm]{Proposition}
\newtheorem{lem}[thm]{Lemma}
\newtheorem{cor}[thm]{Corollary}
\newtheorem{remark}[thm]{Remark}

\theoremstyle{definition}

\theoremstyle{remark}

\newtheorem{rmk}[thm]{Remark}

\catcode`\Ž=\active\def Ž{\'e}
\catcode`\=\active\def {\`e}
\catcode`\ˆ=\active\def ˆ{\`a}
\catcode`\‰=\active\def ‰{\^a}
\catcode`\Ë=\active\def Ë{\`A}
\catcode`\Š=\active\def Š{\"a}
\catcode`\=\active\def {\`u}
\catcode`\=\active\def {\^{e}}
\catcode`\"=\active\def "{\^{i}}
\catcode`\•=\active\def •{\"i}
\catcode`\™=\active\def ™{\^{o}}
\catcode`\š=\active\def š{\"{o}}
\catcode`\ž=\active\def ž{\^{u}}
\catcode`\Ÿ=\active\def Ÿ{\"{u}}
\catcode`\†=\active\def †{\"{U}}
\catcode`\=\active\def {\c{c}}
\catcode`\'=\active\def '{\c{C}}
\catcode`\é=\active\def é{\c{C}}

\newcommand{\shgx}{\Sh(\mathbf G, \mathbf X)}

\newcommand{\ul}{^{(\ell)}}
\newcommand{\uls}{^{(\ell^2)}}

\newcommand{\gofa}{\mathbf G(\mathbf A)}

\newcommand{\qlbar}{\overline{\mathbf Q}_{\ell}}
\newcommand{\glnqlbar}{\GL(n, \qlbar)}
\newcommand{\galf}{ \Gal (\overline{F}/F)}
\newcommand{\rlipi}{R_{\ell, \iota}(\pi)}
\newcommand{\rec}{\rm{rec}}
\newcommand{\sesi}{^{\rm{ss}}}
\newcommand{\Th}{{\rm Th.}}

\newcommand{\Cor}{{\rm Cor.}}

\newcommand{\ord}{{\operatorname{ord }}}

\newcommand{\Sh}{{\operatorname{Sh }}}
\newcommand{\GL}{{\operatorname{GL}}}
\newcommand{\GU}{{\operatorname{GU}}}

\newcommand{\Gal}{{\operatorname{Gal}}}

\newcommand{\calH}{{\mathcal{H}}}

\def\2vector#1#2{\left( \begin{smallmatrix} #1 \\ #2 \end{smallmatrix}
\right)}

\def\deb{ \begin{equation} }
\def\fin{ \end{equation} }

\usepackage[usenames]{color}
\definecolor{Indigo}{rgb}{0.2,0.1,0.7}
\definecolor{Violet}{rgb}{0.5,0.1,0.7}
\definecolor{White}{rgb}{1,1,1}
\definecolor{Green}{rgb}{0.1,0.9,0.2}

\newcommand{\loccit}{{\em loc. cit.}}

\begin{document}

\title{A $\mu$-ordinary Hasse invariant}
\date{\today}
\author{Wushi Goldring, Marc-Hubert Nicole}

\begin{abstract} We construct a generalization of the Hasse invariant for certain unitary Shimura varieties of PEL type whose vanishing locus is the complement of the so-called $\mu$-ordinary locus. We show that the $\mu$-ordinary locus of those varieties is affine. As an application, we strengthen a special case of a theorem of one of us (W.G.) on the association of Galois representations to automorphic representations of unitary groups whose archimedean component is a holomorphic limit of discrete series.

\end{abstract}

\address{W. G. Princeton University, Fine Hall, Washington Road, Princeton, NJ 08544-1000, USA}
\email{wushijig@gmail.com}

\address{M.-H. N.  Institut de mathŽmatiques de Luminy, UniversitŽ d'Aix-Marseille, campus de Luminy, case 907, 13288 Marseille cedex 9, France}
\email{nicole@iml.univ-mrs.fr}

\maketitle

\section{Introduction}
Starting with the cornerstone work of Deligne-Serre \cite{DelSerre} on classical weight one modular forms, the Hasse invariant has been a fundamental tool for constructing congruences between automorphic forms. In turn, the congruences that arise from the Hasse invariant have been used to construct automorphic Galois representations that do not appear directly in the \'etale cohomology of Shimura varieties (\cite{TaySp4}, \cite{WushiGalrepshldsI}). One limitation of the Hasse invariant is that there exist many Shimura varieties for which the Hasse invariant is identically zero. This happens precisely for those Shimura varieties whose ordinary locus is empty.

The $\mu$-ordinary locus introduced by Rapoport and Richartz \cite{RapRic} can be viewed as a substitute to the ordinary locus when the latter is empty. Indeed a theorem of Wedhorn states that, for a prime of good reduction and hyperspecial level, the $\mu$-ordinary locus is dense \cite[\Th 1.6.2]{Wed}. It is therefore natural to seek a generalization of the Hasse invariant whose vanishing locus is the complement of the $\mu$-ordinary locus. We construct such an invariant for Shimura varieties ${\rm Sh}(\mathbf G, \mathbf X)$ of PEL-type such that $\mathbf G(\mathbf R)$ is isomorphic to a unitary similitude group $ \GU(a,b)$ for some positive integers $a,b$. This class includes Picard modular varieties.

\subsection{Main Result} \label{sec main result}
Suppose ${\mathcal U}=(B, V,* <,>, \tilde{h})$ is a Kottwitz datum, with associated Shimura variety $\shgx$  (see \cite[\S3.1]{WushiGalrepshldsI}) such that the center of the simple $\mathbf Q$-algebra $B$ is an imaginary quadratic field $F$. Let $\ell$ be a prime of good reduction for $\mathcal U$ (see \loccit\ \S3.3) and $\mathcal K \ul \subset \mathbf G(\mathbf A_f^{\ell})$ an open compact subgroup. Let $Sh_{\mathcal K \ul, \ell}$ be the Kottwitz integral model of $Sh(\mathbf G, \mathbf X)$ at level $\mathcal K\ul$ (see \loccit\ \S3.4]). Let $E=E(\mathbf G, \mathbf X)$ be the reflex field of $\shgx$ and $\lambda$ a prime of $E$ lying above $\ell$. Denote by $sh_{\mathcal K\ul, \lambda}$ the special fiber of $Sh_{\mathcal K\ul, \ell}$ at $\lambda$ .

\begin{thm}[$\mu$-ordinary Hasse invariant] There exists an automorphic line bundle $^\mu \! \omega_{\mathcal K\ul}$, and a section $^\mu \! H \in H^0(sh_{\mathcal K\ul,\lambda}, ^\mu \!\omega^{\otimes(\ell^{2}-1)}_{\mathcal K\ul})$ such that:
\begin{enumerate}
\item[($\mu$-Ha1)]
The non-vanishing locus of $^\mu \! H$ is the $\mu$-ordinary locus of $sh_{\mathcal K\ul,\lambda}$.
\item[($\mu$-Ha2)]
There exists an integer $m \in \mathbf N$ such that $(^\mu \! H)^m$ lifts to characteristic zero.
\item[($\mu$-Ha3)]
The construction of $^\mu \! H$ is compatible with varying the level $\mathcal K\ul$.

\end{enumerate}

\label{th mu ordinary hasse invariant}\end{thm}

\begin{cor} \label{cor affine}
The $\mu$-ordinary locus $sh_{\mathcal K\ul,\lambda}^{\mu-{\rm ord,min}}$ in the minimal compactification $sh_{\mathcal K\ul,\lambda}^{{\rm min}}$ is affine.
\end{cor}

\begin{remark} {\rm We do not know the minimal value of $m$ in ($\mu$-Ha2). The Hasse invariant of Siegel varieties lifts i.e., $m=1$, see \cite{BocNag}.}
\end{remark}

\subsection{Application.}
Let $\mathcal U$ be as in \S\ref{sec main result}.
Suppose $\pi$ is a cuspidal automorphic representation of $\gofa$ with $p$-adic component $\pi_p$ for every (rational) prime $p$. Given a (rational) prime $\ell$, let ${\mathcal P}\ul$ be the set of (rational) primes $p$ different from $\ell$ such that $\pi_p$ is unramified and $\mathbf G$ is unramified at $p$. Let $\frak P \ul$ be the set of primes of $F$ that are split and lie over some $p \in {\mathcal P}\ul$.

Assume $\wp \in \frak P\ul$. One has a decomposition $\mathbf G(\mathbf Q_p) \cong \mathbf Q_p^{\times} \times \GL(n, F_{\frak \wp})$, where $n$ is given by $n^2=\dim_F{\rm End}_B V$. Write $\pi_p \cong \chi_p \otimes \pi_{\wp}$, with $\chi_p$ a character of $\mathbf Q_p^{\times}$ and $\pi_{\wp}$ a representation of $\GL(n, F_{\wp})$.

Our result on Galois representations is:
\begin{thm}
Suppose $\pi$ is a cuspidal automorphic representation of $\mathbf G(\mathbf A)$ whose archimedean component $\pi_{\infty}$ is an $\mathbf X$-holomorphic limit of discrete series representation of $\mathbf G( \mathbf R)$ (see {\rm \cite[\S2.3]{WushiGalrepshldsI}}). Assume $\ell$ is a prime (of $\mathbf Q$) of good reduction for $\mathcal U$.
Then there exists a unique semisimple Galois representation
\begin{equation}\rlipi:\galf \longrightarrow  \glnqlbar\label{eq gal rep lds} \end{equation}
satisfying the following two conditions:
\begin{enumerate} \item[(Gal1)]
If $p \in \mathcal P\ul$ and $\wp$ is a prime of $F$ dividing $p$ then $\rlipi$ is unramified at $\wp$. In particular $\rlipi$ is unramified at all but finitely many places.

\item[(Gal2)]  If $\wp \in \frak P\ul$ then there is an isomorphism of Weil-Deligne representations
 \begin{equation} (\rlipi | _{W_{F_{\wp}}})^{\sesi}\cong \iota^{-1}\rec(\pi_{\wp} \otimes |\cdot|_{\wp}^{\frac{1-n}{2}}), \end{equation} where $W_{F_{\wp}}$ is the Weil group of $F_{\wp}$, the superscript $\sesi$ denotes semi-simplification and $\rec$ is the Local Langlands Correspondence, normalized as in Harris-Taylor {\rm \cite{HT}}.
 \end{enumerate}
\label{th main unitary 1}
 \end{thm}
\begin{remark} Comparison with {\rm \cite[\Th 1.2.1]{WushiGalrepshldsI}. The improvement in \Th~\ref{th main unitary 1} with respect to {\it loc. cit.} is the removal of the assumption that some prime $\lambda$ of $E$ above $\ell$ is split in $E$.  }\end{remark}

\section{Construction of the $\mu$-ordinary Hasse invariant}
Assume henceforth, without loss of generality, that $a \leq b$. The assumption that $\ell$ is a prime of good reduction for $\mathcal U$ implies that $\ell$ is unramified in $E$.

If $\lambda$ is split in $E$, \Th~\ref{th mu ordinary hasse invariant} is well-known (see e.g., \cite[\S4]{WushiGalrepshldsI}). If $a=b$ then $E=\mathbf Q$, so $\lambda$ is necessarily split in $E$. Hence we assume from now on that $a<b$ and that $\lambda$ is inert in $E$.


As in \cite[\S3.7]{WushiGalrepshldsI}, the Hodge bundle $\Omega_{\mathcal K\ul}$ decomposes as \begin{equation} \Omega_{\mathcal K\ul}=\Omega_{\mathcal K\ul, a}^{\oplus r} \oplus \Omega_{\mathcal K\ul, b}^{\oplus r} \label{eq hodge bundle decomposition}, \end{equation} where $\Omega_{\mathcal K\ul, a}$ (resp. $\Omega_{\mathcal K\ul, b}$) has rank $a$ (resp. $b$) and $r$ is the rank of $B$ over $F$. Let $\omega_{\mathcal K \ul, a}$ (resp. $\omega_{\mathcal K\ul, b}$) be the determinant of $\Omega_{\mathcal K\ul, a}$ (resp. $\Omega_{\mathcal K \ul, b}$).

Let $\mathcal A$ be an abelian scheme representing the universal isogeny class above $sh_{\mathcal K\ul, \lambda}$
As in (4.6) of \loccit\ , the Verschiebung ${\rm Ver:} \ \mathcal A\ul \rightarrow A$ induces a map \begin{equation} {\rm Ver}^*: \Omega_{\mathcal K\ul}  \rightarrow \Omega_{\mathcal K\ul}\ul \label{eq ver hodge bundle}. \end{equation} Since $\lambda$ is inert, the restrictions of ${\rm Ver}^*$ to $\Omega_{\mathcal K\ul,a}$ (resp. $\Omega_{\mathcal K\ul,b}$) have the form \begin{equation} {\rm Ver}^*_{| \Omega_{\mathcal K\ul,a}}:\Omega_{\mathcal K\ul,a} \longrightarrow \Omega_{\mathcal K\ul,b}^{\ell} \mbox{ and } {\rm Ver}^*_{| \Omega_{\mathcal K\ul,b}}:\Omega_{\mathcal K\ul,b} \longrightarrow \Omega_{\mathcal K\ul,a}^{\ell} \label{eq ver permutes a b} . \end{equation}

Therefore, if $({\rm Ver}^*)^2$ denotes the composite of ${\rm Ver}^*$ with itself, then we have \begin{equation} ({\rm Ver}^*)^2_{| \Omega_{\mathcal K\ul, a}}: \Omega_{\mathcal K, a} \longrightarrow \Omega_{\mathcal K, a}\uls \label{eq ver2}.  \end{equation}


Let \begin{equation}  ^{\mu}\!h(\mathcal A): \omega_{\mathcal K\ul, a} \longrightarrow \omega_{\mathcal K\ul, a} ^{\otimes (\ell^{2})} \label{eq mu hasse map}, \end{equation} be the top exterior power of that map, where we have used that $\omega_{\mathcal K\ul, a}^{(\ell^{2})}= \omega_{\mathcal K\ul, a} ^{\otimes (\ell^{2})}$ since $\omega_{\mathcal K\ul, a}$ is a line bundle. The map $^{\mu}\!h(\mathcal A)$ induces a global section \begin{equation} ^{\mu}\!H(\mathcal A) \in H^0(sh_{\mathcal K\ul, \lambda}, \omega_{\mathcal K\ul,a}^{\otimes(\ell^{2}-1)}) \label{eq mu hasse invariant}. \end{equation}

If $\mathcal B$ is another representative of the universal isogeny class above $sh_{\mathcal K\ul, \lambda}$ and $\varphi: \mathcal A \rightarrow \mathcal B$ is an isogeny compatible with the endomorphism actions of $\mathcal A, \mathcal B$, then as in \cite[\S4.2]{WushiGalrepshldsI}, the compatibility of Verschiebung with isogenies (Lemma 4.2.3 of \loccit) implies that $\varphi^*(^\mu\!H(\mathcal B))= {^\mu\!H}(\mathcal A)$. Hence we may omit reference to the representatives $\mathcal A$ or $\mathcal B$ and we have a section $^\mu \!H \in H^0(sh_{\mathcal K \ul, \lambda}, \omega_{\mathcal K\ul,a}^{\otimes(\ell^{2}-1)})$, which we call the $\mu$-{ordinary Hasse invariant}.

\section{Proofs.}
We begin with the proof of \Th~\ref{th mu ordinary hasse invariant}. The following two lemmas and their corollaries will establish that $^\mu\!H$ satisfies ($\mu$-Ha1).

\begin{lem} \label{lemma simpleNP}
 The Newton polygon $\mathcal N^{\rm ord}$ of the underlying isogeny class of abelian schemes of a $\mu$-ordinary geometric point of $sh_{\mathcal K\ul, \lambda}$ has the following slopes:
 \[
\left. \begin{array}{l} 0 \\ 1/2 \\ 1 \end{array}
\right\} \text{with multiplicity } \left\{ \begin{array}{r} 2ar \\ 2(b-a)r \\ 2ar \end{array} \right.
\]

\end{lem}

\begin{proof} The case $r=1$ follows from \cite[2.3.2]{Wed}. The case of general $r$ follows subsequently from \cite[1.3.1 and 3.2.9]{MoonenSerreTate}.
\end{proof}
\begin{prop} The $\mu$-ordinary locus is the maximal $\ell$-rank stratum of $sh_{\mathcal K \ul, \lambda}$. \label{cor mu ord max p rank}
\end{prop}
\begin{proof} The key point is that, by \cite[Prop. 2.4(iv) and \Th 4.2]{RapRic}, the Newton polygon $\mathcal N^{\rm ord}$ described in Lemma~\ref{lemma simpleNP} is the lowest among the Newton polygons of the underlying isogeny classes of abelian schemes corresponding to geometric points of $sh_{\mathcal K\ul, \lambda}$. Let $A$ be an abelian scheme with Newton polygon $\mathcal N(A)$. Then $\mathcal N(A)$ is symmetric and the $\ell$-rank of $A$ is the multiplicity of 0 (=the multiplicity of 1) as a slope of $\mathcal N(A)$.  But if the multiplicity of 0 in $\mathcal N(A)$ is at least the multiplicity of 0 in $\mathcal N^{\rm ord}$ and $\mathcal N(A)$ lies on or above $\mathcal N^{\rm ord}$, then by Lemma~\ref{lemma simpleNP} we must have $\mathcal N(A)=\mathcal N^{\ord}$.
\end{proof}
\begin{cor} The maximal $\ell$-rank stratum of $sh_{\mathcal K \ul, \lambda}$ has $\ell$-rank $2ar$. \label{cor max p rank 2ar} \end{cor}
\begin{proof} This follows directly from Lemma~\ref{lemma simpleNP} and the proof of Prop. ~\ref{cor mu ord max p rank}.
\end{proof}

\begin{lem} Suppose $A$ is an abelian scheme which is a representative of the underlying isogeny class of a geometric point of $sh_{\mathcal K, \lambda}$. Then $^\mu\!H(A) \neq 0$ if and only if the $\ell$-rank of $A$ is equal to $2ar$. \label{lemma hasse p rank}\end{lem}

\begin{proof} One has $H^1(A, \mathcal O_A) \cong H^0(A, \Omega^1_A)$ and under this isomorphism the action of Frobenius on $H^1(A, \mathcal O_A)$ corresponds to that of Verschiebung on $H^0(A, \Omega^1_A)$. Hence \cite[\S15]{MumAbVar} implies that the $\ell$-rank of $A$ equals the semisimple rank of $({\rm Ver}^*)^j:\Omega \rightarrow \Omega^{(\ell^j)}$  for all $j \in \mathbf N$. Since $\dim A=(a+b)r$, keeping in mind~\eqref{eq hodge bundle decomposition} and using the last corollary of \S14 of \loccit, $({\rm Ver}^*)^j$ is semisimple for $j \geq a+b$. Therefore the $\ell$-rank of $A$ equals the rank of $({\rm Ver}^*)^j$ for $j \geq a+b$. We take the $\otimes (a+b)$ iterate of the section $^\mu\!H(A)$, see (\ref{eq mu hasse map}). It is clear that  $^\mu\!H(A) \neq 0$ if and only if  $^\mu\!H(A)^n \neq 0$ for any $n \in \mathbf LN_{>0},$ in particular for $n=a+b$.

Since $a \leq b$, both ${\rm Ver}^*_{|\Omega_{\mathcal K\ul,a}}$ and ${\rm Ver}^*_{|\Omega_{\mathcal K\ul,b}}$ have rank at most $a$. So also $({\rm Ver}^*)^j_{|\Omega_{\mathcal K\ul,a}}$ and $({\rm Ver}^*)^j_{|\Omega_{\mathcal K\ul,b}}$ each have rank at most $a$. By ~\eqref{eq hodge bundle decomposition}, $({\rm Ver^*})^j$ has rank at most $2ar$.

The $\ell$-rank of $A$ equals $2ar$ if and only if the rank of $({\rm Ver}^*)^j$ is $2ar$ for $j \geq a+b$. In turn, the rank of $({\rm Ver}^*)^j$ is $2ar$ if and only if both ${\rm Ver}^*_{|\Omega_{\mathcal K\ul,a}}$ and ${\rm Ver}^*_{|\Omega_{\mathcal K\ul,b}}$ have rank $a$. Since $\Omega_{\mathcal K\ul,a}$ and $\Omega_{\mathcal K\ul,a}^{(\ell^{2(a+b)})}$ are rank $a$ vector bundles, the determinant of a map between them is nonzero if and only if it has rank $a$.
\end{proof}
We now conclude the proof of \Th~\ref{th mu ordinary hasse invariant}.
\begin{proof}[Proof of {\rm \Th~\ref{th mu ordinary hasse invariant}:}]
Combining Prop. ~\ref{cor mu ord max p rank}, \Cor~\ref{cor max p rank 2ar} and Lemma~\ref{lemma hasse p rank} gives ($\mu$-Ha1). By \cite[Prop. 7.14]{LanSuhgen} (or \cite{LanSuhcpt} in the compact case), there exists $k \in \mathbf N$ such that $\omega_{\mathcal K\ul, a}^{\otimes(k(\ell^{2}-1))}$ extends to an ample line bundle on the minimal compactification $sh_{\mathcal K\ul, \lambda}^{\rm min}$. Given this ampleness result, the existence of a lift of some power of $^\mu\!H$ follows by a well-known cohomological argument coupled with the Koecher principle (cf. \cite[Lemma 4.4.1]{WushiGalrepshldsI}). Thus ($\mu$-Ha2) is established. Finally ($\mu$-Ha3) is proved in the same way as \Th 4.2.4 of \loccit . \end{proof}

Next we note that \Cor~\ref{cor affine} is an immediate consequence of \Th~\ref{th mu ordinary hasse invariant}:

\begin{proof}[Proof of {\rm \Cor~\ref{cor affine}:}]
The nonvanishing locus of a section of an ample line bundle on a projective scheme is affine.
\end{proof}
Finally we note how \Th~\ref{th main unitary 1} follows from \Th~\ref{th mu ordinary hasse invariant}.
\begin{proof}[Proof of {\rm \Th~\ref{th main unitary 1}:}] The proof is analogous to the proof of \cite[\Th 1.2]{WushiGalrepshldsI}: Let $\mathcal K \subset \mathbf G(\mathbf A_f)$ be an open compact subgroup such that $\mathcal K=\mathcal K\ul\mathcal K_{(\ell)}$ with $\mathcal K\ul \subset \mathbf G(\mathbf A_f^\ell)$ and $\mathcal K_{(\ell)} \subset \mathbf G(\mathbf Z_{\ell})$. Let $Sh_{\mathcal K, E}$ be the model of $\shgx$ at level $\mathcal K$ over $E$ as defined in \S3.2 of \loccit. Let $Sh_{\mathcal K, \ell}$ be the normalization of $Sh_{\mathcal K\ul, \ell}$ in $Sh_{\mathcal K, E}$. Let ${\rm \pi}:Sh_{\mathcal K, \ell} \rightarrow Sh_{\mathcal K\ul, \ell}$ be the natural projection.

Using ($\mu$-Ha2), let $^\mu\!H_{\mathcal K\ul}^{\rm lift}$ be a lift of a power $^\mu\!H$ and let $^\mu\!H_{\mathcal K}^{\rm lift}$ be the pullback of $^\mu\!H_{\mathcal K\ul}^{\rm lift}$ to $Sh_{\mathcal K,\ell}$ along the projection $\pi$.

Let $\mathcal H=\mathcal H_{\mathcal P \ul}(\mathbf G, \mathbf Z_\ell)$ be the spherical Hecke algebra of $\mathbf G$ with values in $\mathbf Z_\ell$, trivial at places outside $\mathcal P\ul$ (see \S6.1 of \loccit\  for a more detailed definition).
\begin{thm} \label{th mu hasse commutes with hecke}
 Suppose $\mathcal V$ is an automorphic vector bundle on $Sh_{\mathcal K, \ell}$ and $f \in H^0(Sh_{\mathcal K, \ell}, \mathcal V)$ is nonzero modulo $\lambda$. Then for all $j \in \mathbf N$, the product $(^\mu\!H_{\mathcal K}^{\rm lift})^{\ell^j}f \in H^0(Sh_{\mathcal K, \ell}, \omega_{\mathcal K, a}^{\otimes (\ell^{2}-1)\ell^j} \otimes \mathcal V)$ is nonzero modulo $\lambda$ and satisfies \begin{equation} T((^\mu\!H_{\mathcal K}^{\rm lift})^{\ell^j}f) \equiv (^\mu\!H_{\mathcal K}^{\rm lift})^{\ell^j}T(f) \pmod{\lambda^{j+1}} \mbox{ for all } T \in \calH. \label{eq mu hasse commutes with hecke} \end{equation} \end{thm}
 \begin{proof} Since the $\mu$-ordinary locus is dense \cite[\Th 1.6.2]{Wed}, the product $(^\mu\!H_{\mathcal K}^{\rm lift})^{\ell^j}f$ is nonzero modulo $\lambda$ by ($\mu$-Ha1). As in the proof of \cite[\Th 6.2.1]{WushiGalrepshldsI}, ($\mu$-Ha3) implies~\eqref{eq mu hasse commutes with hecke}. \end{proof}
 Given \Th~\ref{th mu hasse commutes with hecke}, the remainder of the argument to establish \Th~\ref{th main unitary 1} is identical to \S\S6.3-6.4 of \loccit.
\end{proof}

\begin{rmk} A tremendous advantage of our $\mu$-ordinary Hasse invariant is that it satisfies all key properties of the classical invariant. Its applications will thus follow the classical blueprint: to Galois representations (as we illustrated briefly above), but also immediately the (non-effective) existence of its canonical subgroup thanks to the elementary \cite[Prop.3]{Fargues}, and thus also applications to explicit constructions of eigenvarieties, etc.

\end{rmk}

\section{Acknowledgments}

We thank the Max-Planck-Institut fŸr Mathematik for a year-long membership in 2011 (M.-H. N.) and also for a short visit in May 2011 (W.G.). In particular, the natural albeit key idea of considering higher powers of Verschiebung occurred to M.-H. N. on his very first Monday at MPIM.

W.G. thanks Pierre Deligne, Richard Taylor, Barry Mazur, Elena Mantovan, David Geraghty, Beno\^\i t Stroh, Jacques Tilouine and Vincent Pilloni for helpful conversations. W. G. is happy to acknowledge support from 10 BLAN 114 01 ANR ARSHIFO, a Simons Travel Grant and NSF MSPRF.

\bibliographystyle{amsalpha}
\bibliography{galrepsholodsunitary}
\end{document}